\documentclass[a4paper]{amsart}
\usepackage{graphicx}
\usepackage{amssymb}
\usepackage{amsmath}
\usepackage{amsthm}
\usepackage{amscd}
\usepackage[all,2cell]{xy}

\UseAllTwocells \SilentMatrices
\newtheorem{thm}{Theorem}[section]

\newtheorem{lem}[thm]{Lemma}

\newtheorem{prop}[thm]{Proposition}
\theoremstyle{definition}
\newtheorem{defn}[thm]{Definition}
\theoremstyle{remark}
\newtheorem{rem}[thm]{\bf Remark}
\numberwithin{equation}{section}

\begin{document}
\title[Weighted projective lines of tubular type and equivariantization]{Weighted projective lines of tubular type and equivariantization}

\author[Jianmin Chen, Xiao-Wu Chen] {Jianmin Chen, Xiao-Wu Chen$^*$}

\thanks{$^*$ the corresponding author}
\date{\today}
\maketitle

\dedicatory{}%
\commby{}%

\begin{abstract}
 We prove that the categories of coherent sheaves over  weighted projective lines of tubular type are explicitly related to each other via the equivariantization with respect to certain cyclic group actions.
\end{abstract}

\section{Introduction}

The notion of weighted projective lines is introduced in \cite{GL87,GL90}, which gives a geometric treatment to the representation theory of the canonical algebras in the sense of \cite{Rin}. We are interested in weighted projective lines of tubular type. Recall that the category of coherent sheaves over such a weighted projective line is derived equivalent to the module category over a tubular algebra of the same type.

It is known due to \cite{GL87, Len, CCZ} that the category of coherent sheaves over a weighted projective line of tubular type is equivalent to the category of equivariant coherent sheaves over an elliptic curve with respect to a certain cyclic group action; compare \cite{Po}. This result explains well that the classification of indecomposable modules over a tubular algebra in \cite{Rin,LM} has similar features as the classification of indecomposable coherent sheaves over an elliptic curve in \cite{At}.

In this paper, we show that the categories of coherent sheaves over weighted projective lines of different tubular types are related to each other, via the equivariantization with respect to certain cyclic group actions. Indeed, these cyclic groups are of order two or three, and the actions are the degree-shift actions,  which are induced from the grading on  the homogeneous coordinate algebras. Here, the equivariantization means forming the category of equivariant objects for a given finite group action on a category; compare \cite{RR, De, DGNO}.

Let us describe the main results of this paper. Let $k$ be an algebraically closed field, whose characteristic is different from two or three. According to the types, weighted projective lines $\mathbb{X}$ of tubular type are denoted by $\mathbb{X}(2,2,2,2; \lambda)$, $\mathbb{X}(3,3,3)$, $\mathbb{X}(4,4,2)$ and $\mathbb{X}(6,3,2)$, respectively. Here, $\lambda\in k$ is not $0$ or $1$. The Auslander-Reiten translation on the category ${\rm coh}\mbox{-}\mathbb{X}$  of coherent sheaves  over $\mathbb{X}$ is induced from the degree-shift automorphism by the dualizing element $\vec{\omega}$, which is an element in the grading group of the homogeneous coordinate algebra of $\mathbb{X}$.

 In the tubular types, the dualizing element $\vec{\omega}$ has order $2$, $3$, $4$ and $6$, according to their types. By the degree-shift automorphisms, we have a strict action on ${\rm coh}\mbox{-}\mathbb{X}$ by the cyclic group $\mathbb{Z}\vec{\omega}$ and also by its subgroup. For a finite group $G$ and a (strict) $G$-action on a category $\mathcal{A}$, we denote by $\mathcal{A}^G$ the category of equivariant objects. In particular, we have the category $({\rm coh}\mbox{-}\mathbb{X})^G$ for any subgroup $G\subseteq \mathbb{Z}\vec{\omega}$.

 The following theorem combines Propositions \ref{prop:A}, \ref{prop:B} and \ref{prop:C}. Here, we fix $\epsilon \in k$ satisfying $\epsilon^2-\epsilon+1=0$.

\vskip 10pt

\noindent {\bf Theorem.} \emph{
Keep the notation and assumptions as above. Denote by $\vec{\omega}$ the dualizing element in the grading group of the homogeneous coordinate algebra of $\mathbb{X}$. Then we have the following equivalences of categories.
\begin{enumerate}
\item $({\rm coh}\mbox{-}\mathbb{X}(4,4,2))^{\mathbb{Z}(2\vec{\omega})}\stackrel{\sim}\longrightarrow {\rm coh}\mbox{-}\mathbb{X}(2,2,2,2;-1).$
\item $({\rm coh}\mbox{-}\mathbb{X}(6,3,2))^{\mathbb{Z}(2\vec{\omega})}\stackrel{\sim}\longrightarrow {\rm coh}\mbox{-}\mathbb{X}(2,2,2,2;\epsilon).$
\item $({\rm coh}\mbox{-}\mathbb{X}(6,3,2))^{\mathbb{Z}(3\vec{\omega})}\stackrel{\sim}\longrightarrow {\rm coh}\mbox{-}\mathbb{X}(3,3,3).$ \hfill $\square$
\end{enumerate}
}

\vskip 10pt

The acting groups in (1) and (3) are of order two, and the one in (2) is of order three. For the proof, we construct explicit algebra homomorphisms between the corresponding homogeneous coordinate algebras. In all these three cases, the algebra homomorphisms are injective, whose images equal certain restriction subalgebras \cite{CCZ} of the homogeneous coordinate algebras. We mention that the idea using the restriction subalgebras goes back to \cite{GL87, GL90, Len94}.

The paper is structured as follows. In Section 2, we recall basic facts on the homogeneous coordinate algebras of weighted projective lines, and their restriction subalgebras. We introduce admissible homomorphisms between the string groups, which play a role in Proposition \ref{prop:equiv}. Indeed, Proposition \ref{prop:equiv} claims that under certain conditions, the categories of coherent sheaves over different weighted projective lines are related via the equivariantization with respect to a degree-shift action.

In Section 3, we prove Theorem. For the proof, we construct explicit admissible homomorphisms between the string groups and algebra homomorphisms between the homogeneous coordinate algebras, which are verified to satisfy the conditions in Proposition \ref{prop:equiv}. In Proposition \ref{prop:D}, we show that the categories of coherent sheaves over weighted projective lines of the same weight type $(2,2,2,2)$ but with different parameters,  might be related to each other via the equivariantization with respect to a certain degree-shift action.

\section{The restriction subalgebras and equivariantization}

In this section, we recall from \cite{GL87, GL90} basic facts on the homogeneous coordinate algebra of a weighted projective line. We study admissible homomorphisms between the string groups and algebra homomorphisms between the homogeneous coordinate algebras, whose images equal the restriction subalgebras with respect to an effective subgroup of the string group. We formulate Proposition \ref{prop:equiv}, which gives sufficient conditions on when two categories of coherent sheaves over different weighted projective lines are related by the equivariantization with respect to a certain degree-shift action.

\subsection{} Let $t\geq 1$ be an integer. A \emph{weight sequence} ${\bf p}=(p_1, p_2, \cdots, p_t)$ of length  $t$ consists of positive integers satisfying $p_i\geq 2$.  We might assume that the weight sequence ${\bf p}$ satisfies $p_1\geq p_2\geq \cdots \geq p_t$.

The \emph{string group} $L({\bf p})$ is an abelian group with generators $\vec{x}_1, \vec{x}_2, \cdots, \vec{x}_t$ subject to the relations $p_1\vec{x}_1=p_2\vec{x}_2=\cdots = p_t\vec{x}_t$, where this common element is denoted by $\vec{c}$ and called the \emph{canonical element}. Here, the abelian group is written additively.

The string group $L(\mathbf{p})$ is of rank one, where $\vec{c}$ is of infinite order. There is an isomorphism of abelian groups
$$ L(\mathbf{p})/\mathbb{Z}\vec{c}\stackrel{\sim}\longrightarrow \prod_{i=1}^t \mathbb{Z}/p_i\mathbb{Z},$$
sending $\vec{x}_i+\mathbb{Z}\vec{c}$ to the vector $(0, \cdots,0,  \bar{1}, 0, \cdots, 0)$ with $\bar{1}$ on the $i$-th component. From this isomorphism, we deduce that each element $\vec{x}$ in $L(\mathbf{p})$ is uniquely written in its \emph{normal form}
\begin{align}\label{equ:nor}
\vec{x}=l\vec{c}+\sum_{i=1}^t l_i\vec{x}_i
\end{align}
with $l\in \mathbb{Z}$ and $0\leq l_i\leq p_i-1$.

For each $1\leq i\leq t$, we have a surjective group homomorphism $\pi_i\colon L(\mathbf{p})\rightarrow \mathbb{Z}/p_i\mathbb{Z}$ with $\pi_i(\vec{x}_j)=\delta_{i,j}\bar{1}$. Here, $\delta$ is the Kronecker symbol. Following \cite[Definition 6.5]{CCZ}, an infinite subgroup $H\subseteq L(\mathbf{p})$ is \emph{effective} provided that $\pi_i(H)=\mathbb{Z}/p_i\mathbb{Z}$, or equivalently, $\bar{1}$ lies in $\pi_i(H)$ for each $1\leq i\leq t$.

Let $p={\rm lcm}(\mathbf{p})={\rm lcm}(p_1, p_2, \cdots, p_t)$ be the least common multiple of $\mathbf{p}$. There is a unique surjective group homomorphism $\delta\colon L(\mathbf{p})\rightarrow \mathbb{Z}$, called the \emph{degree map}, given by $\delta(\vec{x}_i)=\frac{p}{p_i}$. We observe that the kernel of $\delta$ equals the torsion subgroup of $L(\mathbf{p})$. Recall that the \emph{dualizing element} $\vec{\omega}$ in $L(\mathbf{p})$ is defined as $\vec{\omega}=(t-2)\vec{c}-\sum_{i=1}^t \vec{x}_i$. We observe that $\delta(\vec{\omega})=p((t-2)-\sum_{i=1}^t\frac{1}{p_i})$.

The weight sequence $\mathbf{p}$ is of \emph{tubular type} provided that $\vec{\omega}$ is a torsion element, or equivalently, $\delta(\vec{\omega})=0$. From this,  we infer that $\mathbf{p}=(2,2,2,2)$, $(3,3,3)$, $(4,4,2)$  and $(6,3,2)$.

We mention that if the weight sequence $\mathbf{p}$ is of non-tubular type, or equivalently, the dualizing element  $\vec{\omega}$ is of infinite order, then the subgroup $\mathbb{Z}\vec{\omega}\subseteq L(\mathbf{p})$ is effective.

For each element $\vec{x}=l\vec{c}+\sum_{i=1}^t l_i\vec{x}_i$ in its normal form, we set ${\rm mult}(\vec{x})={\rm max}\{l+1, 0\}$. This gives rise to a map
\begin{align}
{\rm mult}\colon L(\mathbf{p})\longrightarrow \mathbb{N}=\{0, 1, 2, \cdots\}.
\end{align}

Let $\mathbf{q}=(q_1,q_2, \cdots, q_s)$ be another weight sequence and $L(\mathbf{q})$ be the corresponding string group. Then we have the map ${\rm mult}\colon L(\mathbf{q})\rightarrow \mathbb{N}$ as above.

\begin{defn}\label{defn:AH}
A group homomorphism $\pi\colon L(\mathbf{q})\rightarrow L(\mathbf{p})$ is \emph{admissible} provided that the following conditions are satisfied
\begin{enumerate}
\item[(AH1)]  the subgroup ${\rm Im}\pi\subseteq L(\mathbf{p})$ is effective;
\item[(AH2)]  for each $\vec{x}\in {\rm Im}\pi$, we have $\sum_{\vec{y}\in \pi^{-1}(\vec{x})} {\rm mult}(\vec{y})={\rm mult}(\vec{x})$. \hfill $\square$
\end{enumerate}
\end{defn}

We observe that  by (AH1) the fiber $\pi^{-1}(\vec{x})$ for each $\vec{x}\in {\rm Im}\pi$ is a finite set.

\subsection{}

Let $k$ be an arbitrary  field. A \emph{parameter sequence} $\mathbf{\lambda}=(\lambda_1, \lambda_2, \cdots, \lambda_t)$ of length $t$ consists of a collection of pairwise distinct rational points on the projective line $\mathbb{P}_k^1$. The parameter sequence is \emph{normalized} if $\lambda_1=\infty$, $\lambda_2=0$ and $\lambda_3=1$.

 A \emph{weighted projective line} $\mathbb{X}=\mathbb{X}(\mathbf{p}, \mathbf{\lambda})$ of weight type $\mathbf{p}$ and parameter sequence $\mathbf{\lambda}$ is by definition the projective line $\mathbb{P}_k^1$ such that each point $\lambda_i$ has weight $p_i$. We will assume that the parameter sequence $\mathbf{\lambda}$ is normalized.  The \emph{homogeneous coordinate algebra} $S=S(\mathbf{p}, \mathbf{\lambda})$ of the weighted projective line $\mathbb{X}$ is defined to be $k[X_1, X_2, \cdots, X_t]/I$, where the ideal $I$ is generated by $X_i^{p_i}-(X_2^{p_2}-\lambda_iX_1^{p_1})$ for $3\leq i\leq t$. We write $x_i=X_i+I$ in $S$.

The algebra $S$ is $L(\mathbf{p})$-graded by means of ${\rm deg}x_i=\vec{x}_i$. Then we have $S=\bigoplus_{\vec{x}\in L(\mathbf{p})} S_{\vec{x}}$, where $S_{\vec{x}}$ denotes the homogeneous component of degree $\vec{x}$. Write $\vec{x}=l\vec{c}+\sum_{i=1}^tl_i\vec{x}_i$ in its normal form (\ref{equ:nor}). Then $S_{\vec{x}}\neq 0$ if and only if $l\geq 0$. Indeed, $\{x_1^{ap_1}x_2^{bp_2}x_1^{l_1}x_2^{l_2}\cdots x_t^{l_t}\; |\; a+b=l, a, b\geq 0\}$ is a basis of $S_{\vec{x}}$; compare \cite[Proposition 1.3]{GL87}. We deduce that
\begin{align}\label{equ:dim}
{\rm dim}_k\; S_{\vec{x}}={\rm mult}(\vec{x}), \mbox{ for all } \vec{x}\in L(\mathbf{p}).
\end{align}

For an infinite subgroup $H\subseteq L(\mathbf{p})$, we have the \emph{restriction subalgebra} $S_H=\bigoplus_{\vec{x}\in H}S_{\vec{x}}$ of $S$, which is viewed as an $H$-graded algebra. We recall from \cite[Lemma 6.2]{CCZ}  that this algebra $S_H$ is a finitely generated $k$-algebra.

 In general, the structure of these restriction subalgebras $S_H$ is not known. We mention that if $\mathbf{p}$ is of non-tubular type and $H=\mathbb{Z}\vec{\omega}$, the corresponding restriction subalgebras are related to Kleinian singularities and Fuchsian singularities; see \cite[Proposition 8.4]{GL90} and \cite[Sections 5 and 6]{Len94}. In what follows, we explain the importance of effective subgroups $H$ in $L(\mathbf{p})=L$.

We recall that the weighted projective line $\mathbb{X}$ is endowed with a structure sheaf of $L$-graded commutative noetherian algebras. Then the abelian category ${\rm coh}\mbox{-}\mathbb{X}$ of ($L$-graded) coherent sheaves on $\mathbb{X}$ is defined. In what follows, we will recall a more convenient description of  ${\rm coh}\mbox{-}\mathbb{X}$ via graded $S$-modules.

We denote by ${\rm mod}^L\mbox{-}S$ the abelian category of finitely generated $L$-graded $S$-modules, and by ${\rm mod}_0^L\mbox{-}S$ its Serre subcategory formed by finite dimensional modules. We denote by ${\rm qmod}^L\mbox{-}S={\rm mod}^L\mbox{-}S/{{\rm mod}_0^L\mbox{-}S}$ the quotient abelian category. By \cite[Theorem 1.8]{GL87} the sheafification functor yields an equivalence
\begin{align}\label{equ:coh}
{\rm qmod}^L\mbox{-}S\stackrel{\sim}\longrightarrow {\rm coh}\mbox{-}\mathbb{X}.
\end{align}
We will identify these two categories.

Each $\vec{x}\in L$ gives rise to an automorphism $(\vec{x})\colon {\rm mod}^L\mbox{-}S\rightarrow {\rm mod}^L\mbox{-}S$, called the \emph{degree-shift functor}, as follows. For each $L$-graded $S$-module $M=\bigoplus_{\vec{l}\in L}M_{\vec{l}}$, the new module $M(\vec{x})=M$ is graded by $M(\vec{x})_{\vec{l}}=M_{\vec{l}+\vec{x}}$.  This degree-shift functor induces the corresponding automorphisms on ${\rm qmod}^L\mbox{-}S$ and on ${\rm coh}\mbox{-}\mathbb{X}$, both of which are denoted by $(\vec{x})$.

For an infinite subgroup $H\subseteq L(\mathbf{p})$, the restriction subalgebra $S_H$ is $H$-graded. Then we have abelian categories ${\rm mod}^H\mbox{-}S_H$ and ${\rm qmod}^H\mbox{-}S_H$. The \emph{restriction functor} ${\rm res}\colon {\rm mod}^L\mbox{-}S\rightarrow {\rm mod}^H\mbox{-}S_H$ sends an $L$-graded $S$-module $M=\bigoplus_{\vec{l}\in L}M_{\vec{l}}$ to $M_H=\bigoplus_{\vec{l}\in H} M_{\vec{l}}$, which is naturally an $H$-graded $S_H$-module. The exact functor ``res" preserves finite dimensional modules, and induces an exact functor between the quotient categories
$${\rm res}\colon {\rm qmod}^L\mbox{-}S\longrightarrow {\rm qmod}^H\mbox{-}S_H.$$

The following result combines \cite[Proposition 6.6]{CCZ} with the equivalence (\ref{equ:coh}).

\begin{lem}\label{lem:eff}
Let $H\subseteq L=L(\mathbf{p})$ be an infinite subgroup. Then the induced functor ${\rm res}\colon {\rm qmod}^L\mbox{-}S\rightarrow {\rm qmod}^H\mbox{-}S_H$ is an equivalence if and only if the subgroup $H$ is effective, in which case, we might identify ${\rm coh}\mbox{-}\mathbb{X}$ with ${\rm qmod}^H\mbox{-}S_H$.  \hfill $\square$
\end{lem}

\subsection{}

Let $\mathbf{q}=(q_1, q_2, \cdots, q_s)$ and $\mathbf{\mu}=(\mu_1, \mu_2, \cdots, \mu_s)$ be another weight sequence and parameter sequence, respectively. We have the homogeneous coordinate algebra $S(\mathbf{q}, \mathbf{\mu})$ of the weighted projective line $\mathbb{X}(\mathbf{q}, \mathbf{\mu})$.

We suppose that $\pi\colon L(\mathbf{q})\rightarrow L(\mathbf{p})$ is a group homomorphism with ${\rm Im}\pi\subseteq L(\mathbf{p})$ an infinite subgroup, and that there is an algebra homomorphism $\phi\colon S(\mathbf{q}, \mathbf{\mu})\rightarrow S(\mathbf{p}, \mathbf{\lambda})$  satisfying $\phi(S(\mathbf{q}, \mathbf{\mu})_{\vec{y}})\subseteq S(\mathbf{p}, \mathbf{\lambda})_{\pi(\vec{y})}$ for each $\vec{y}\in L(\mathbf{q})$. We observe that $\phi$ induces an ${\rm Im}\pi$-graded algebra homomorphism
\begin{align}\label{equ:phi}
\bar{\phi}\colon \pi_*S(\mathbf{q}, \mathbf{\mu})\longrightarrow S(\mathbf{p}, \mathbf{\lambda})_{{\rm Im}\pi}.
\end{align}
Here, $S(\mathbf{p}, \mathbf{\lambda})_{{\rm Im}\pi}$ is the restriction subalgebra with respect to the subgroup ${\rm Im}\pi\subseteq L(\mathbf{p})$. For the ${\rm Im}\pi$-graded algebra $\pi_*S(\mathbf{q}, \mathbf{\mu})$, we apply the notation in the following paragraph to the surjective homomorphism $L(\mathbf{q})\rightarrow {\rm Im}\pi$.

Given a surjective group homomorphism $\theta\colon G\rightarrow K$ and a $G$-graded algebra $A=\bigoplus_{g\in G}A_G$, we define a $K$-graded algebra $\theta_*A$ as follows: as an ungraded algebra $\theta_*A=A$, while its homogeneous component $(\theta_*A)_h=\bigoplus_{g\in \theta^{-1}(h)} A_g$ for each $h\in K$. In other words, $\theta_*A$ equals $A$, but with a coarser grading.

\begin{lem}\label{lem:phi}
Keep the notation as above. Assume that the ${\rm Im}\pi$-graded algebra homomorphism $\bar{\phi}$ is surjective and that $\pi\colon L(\mathbf{q})\rightarrow L(\mathbf{p})$ satisfies ${\rm (AH2)}$ in Definition {\rm \ref{defn:AH}}. Then $\bar{\phi}$ is an isomorphism.
\end{lem}

\begin{proof}
Let $\vec{x}\in {\rm Im}\pi$. We apply (\ref{equ:dim}) and (AH2) to infer that ${\rm dim}_k\;  (\pi_*S(\mathbf{q}, \mathbf{\mu}))_{\vec{x}}={\rm dim}_k \; S(\mathbf{p}, \mathbf{\lambda})_{\vec{x}}$, both of which are finite. By the surjectivity assumption, we infer that  $\bar{\phi}$ is bijective on each degree. Then we are done.
\end{proof}

We recall from \cite{DGNO,De,CCZ} the equivariantization briefly. Let $G$ be a group with unit $e$ and $\mathcal{A}$ be a category.  A \emph{strict} $G$-action on $\mathcal{A}$ is a group homomorphism from $G$ to the automorphism group of $\mathcal{A}$, which assigns for each $g\in G$ an automorphism $F_g$ on $\mathcal{A}$. Hence, we have $F_e={\rm Id}_\mathcal{A}$ and $F_g F_h=F_{gh}$. Temporarily, we write $G$ mutliplicatively.

A \emph{$G$-equivariant object} in $\mathcal{A}$ is a pair $(X, \alpha)$, where $X$ is an object in $\mathcal{A}$ and $\alpha$ assigns to each $g\in G$ an isomorphism $\alpha_g\colon X\rightarrow F_g(X)$ subject to the relations $\alpha_{gh}=F_g(\alpha_h)\circ \alpha_g$. A morphism $f\colon (X, \alpha)\rightarrow (Y, \beta)$ between equivariant objects is  a morphism $f\colon X\rightarrow Y$ in $\mathcal{A}$ satisfying $\beta_g\circ f=F_g(f)\circ \alpha_g$. This gives rise to the category $\mathcal{A}^G$ of equivariant objects.

We observe if $\mathcal{A}$ is abelian, so is $\mathcal{A}^G$. Indeed, a sequence of equivariant objects is exact in $\mathcal{A}^G$ if and only if so is the sequence of underlying objects in $\mathcal{A}$.

For each subgroup $N\subseteq L(\mathbf{q})$, we have a strict $N$-action on ${\rm mod}^{L(\mathbf{q})}\mbox{-}S(\mathbf{q}, \mathbf{\mu})$ by setting $F_{\vec{x}}=(-\vec{x})$ for each $\vec{x}\in N$. Here, $(-\vec{x})$ is the degree-shift functor  by the element $-\vec{x}$. This $N$-action induces strict $N$-actions on both ${\rm qmod}^{L(\mathbf{q})}\mbox{-}S(\mathbf{q}, \mathbf{\mu})$ and ${\rm coh}\mbox{-}\mathbb{X}(\mathbf{q}, \mathbf{\mu})$. These resulted $N$-actions are called the \emph{degree-shift actions}. In particular, we will consider the categories $({\rm qmod}^{L(\mathbf{q})}\mbox{-}S(\mathbf{q}, \mathbf{\mu}))^N$ and $({\rm coh}\mbox{-}\mathbb{X}(\mathbf{q}, \mathbf{\mu}))^N$ of $N$-equivariant objects.

\begin{prop}\label{prop:equiv}
Let $\pi\colon L(\mathbf{q})\rightarrow L(\mathbf{p})$ be an admissible homomorphism. Assume that the algebra homomorphism $\phi\colon S(\mathbf{q}, \mathbf{\mu})\rightarrow S(\mathbf{p}, \mathbf{\lambda})$ induces a surjective homomorphism $\bar{\phi}$ in {\rm (\ref{equ:phi})}. Then $\bar{\phi}\colon \colon \pi_*S(\mathbf{q}, \mathbf{\mu})\rightarrow S(\mathbf{p}, \mathbf{\lambda})_{{\rm Im}\pi}$ is an isomorphism of ${\rm Im}\pi$-graded algebras, and thus we have an equivalence of categories
\begin{align*}
({\rm coh}\mbox{-}\mathbb{X}(\mathbf{q}, \mathbf{\mu}))^{{\rm Ker}\pi}\stackrel{\sim}\longrightarrow {\rm coh}\mbox{-}\mathbb{X}(\mathbf{p}, \mathbf{\lambda}).
\end{align*}
\end{prop}

\begin{proof}
We identify
$({\rm coh}\mbox{-}\mathbb{X}(\mathbf{q}, \mathbf{\mu}))^{{\rm Ker}\pi}$ and $({\rm qmod}^{L(\mathbf{q})}\mbox{-}S(\mathbf{q}, \mathbf{\mu}))^{{\rm Ker}\pi}$ via (\ref{equ:coh}). By \cite[Proposition 5.2 and Corollary 4.4]{CCZ}, these two categories are further equivalent to ${\rm qmod}^{{\rm Im}\pi}\mbox{-}\pi_*S(\mathbf{q}, \mathbf{\mu})$, which is isomorphic to ${\rm qmod}^{{\rm Im}\pi}\mbox{-}S(\mathbf{p}, \mathbf{\lambda})_{{\rm Im}\pi}$. Here, we use the isomorphism $\bar{\phi}$ in Lemma \ref{lem:phi}. Since the subgroup ${\rm Im}\pi\subseteq L(\mathbf{p})$ is effective, we are done by the identification in Lemma \ref{lem:eff}.
\end{proof}

\section{The proof of Theorem}

In this section, we study the homogeneous coordinate algebras of weighted projective lines of tubular type. We  prove Theorem in Propositions \ref{prop:A}, \ref{prop:B} and \ref{prop:C}.  We construct explicit admissible homomorphisms between the string groups and algebra homomorphisms between the homogeneous coordinate algebras, which satisfy the conditions in  Proposition \ref{prop:equiv}.

\subsection{}

We assume that $\mathbf{p}=(p_1, p_2, \cdots, p_t)$ is a weight sequence of tubular type and that $S(\mathbf{p}, \mathbf{\lambda})$ is the homogeneous coordinate algebra of the weighted projective line $\mathbb{X}(\mathbf{p}, \mathbf{\lambda})$. Here, $\mathbf{\lambda}=(\lambda_1, \lambda_2, \cdots, \lambda_t)$ is a normalized parameter sequence. Then we have $\mathbf{p}=(2,2,2,2)$, $(3,3,3)$, $(4,4,2)$ or $(6,3,2)$. Since the parameter sequence is normalized, it is trivial if the length of $\mathbf{p}$ is three. Hence,  we  might write $S(\mathbf{p}, \mathbf{\lambda})$ as $S(2,2,2,2;\lambda)$, $S(3,3,3)$, $S(4,4,2)$ or $S(6,3,2)$, according to their \emph{types}. Here, the scalar $\lambda\in k$ is not $0$ or $1$.

We list these homogeneous coordinate algebras explicitly as follows.
\begin{align*}
S(2,2,2,2; \lambda)&=k[X_1, X_2, X_3, X_4]/(X_3^2-(X_2^2-X_1^2), X_4^2-(X_2^2-\lambda X_1^2));\\
S(3,3,3)&=k[Y_1, Y_2, Y_3]/(Y_3^3-(Y_2^3-Y_1^3));\\
S(4,4,2)&=k[Z_1, Z_2, Z_3]/(Z_3^2-(Z_2^4-Z_1^4));\\
S(6,3,2)&=k[U_1, U_2, U_3]/(U_3^2-(U_2^3-U_1^6)).
\end{align*}
Here, we use different letters for the generators to avoid confusion. Moreover, we will use letters in the lower case to represent their images in the quotient algebras. For example, $y_i$ will represent the image of $Y_i$ in $S(3,3,3)$.

\subsection{} In this subsection, we will relate weighted projective lines $\mathbb{X}(2,2,2,2; -1)$ to $\mathbb{X}(4,4,2)$. Here, we require that the field $k$ is not of characteristic two.

We consider the corresponding string groups. Recall that $L(2,2,2,2)$ is generated by $\vec{x}_1, \vec{x}_2, \vec{x}_3$ and $\vec{x}_4$ with the relations $2\vec{x}_1=2\vec{x}_2=2\vec{x}_3=2\vec{x}_4$. The string group $L(4,4,2)$ is generated by $\vec{z}_1, \vec{z}_2$ and $\vec{z}_3$ with the relations $4\vec{z}_1=4\vec{z}_2=2\vec{z}_3$. Then we have a well-defined group homomorphism $\pi\colon L(4,4,2)\rightarrow L(2,2,2,2)$ by $\pi(\vec{z}_1)=\vec{x}_1$, $\pi(\vec{z}_2)=\vec{x}_2$ and $\pi(\vec{z}_3)=\vec{x}_3+\vec{x}_4$.

\begin{lem}\label{lem:A}
The above defined group homomorphism $\pi\colon L(4,4,2)\rightarrow L(2,2,2,2)$ is admissible with ${\rm Ker}\pi=\{0, 2\vec{z}_1+2\vec{z}_2-\vec{c}\}=\mathbb{Z}(2\vec{\omega})$.
\end{lem}

Here, we recall that the dualizing element $\vec{\omega}=\vec{c}-\vec{z}_1-\vec{z}_2-\vec{z}_3$ in $L(4,4,2)$ has order four. The cyclic subgroup generated by $2\vec{\omega}=2\vec{z}_1+2\vec{z}_2-\vec{c}$ is denoted by $\mathbb{Z}(2\vec{\omega})$.

\begin{proof}
We observe that ${\rm Im}\pi\subseteq L(2,2,2,2)$ is effective. Write an element $\vec{x}\in L(2,2,2,2)$  in its normal form $\vec{x}=l\vec{c}+l_1\vec{x}_1+l_2\vec{x}_2+l_3\vec{x}_3+l_4\vec{x}_4$ with each $l_i\in\{0, 1\}$. We observe that $\vec{x}$ lies in ${\rm Im}\pi$ if and only if $l_3=l_4$.

For an element $\vec{z}=r\vec{c}+r_1\vec{z}_1+r_2\vec{z}_2+r_3\vec{z}_3\in L(4,4,2)$ in its normal form, we have $\pi(\vec{z})=2r\vec{c}+r_1\vec{x}_1+r_2\vec{x}_2+r_3\vec{x}_3+r_3\vec{x}_4$. This expression might not be a normal form. Indeed, it depends on whether $r_1$ and $r_2$ are larger than two or not.

We now prove (AH2) for the above $\vec{x}\in {\rm Im}\pi$.  Indeed, if $l$ is even, we infer from the above analysis that $\pi^{-1}(\vec{x})=\{\frac{l}{2}\vec{c}+l_1\vec{z}_1+l_2\vec{z}_2+l_3\vec{z}_3, \frac{l-2}{2}\vec{c}+(l_1+2)\vec{z}_1+(l_2+2)\vec{z}_2+l_3\vec{z}_3\}$, which implies the statement on the kernel of $\pi$. Then the required identity in (AH2) follows by
\begin{align}\label{equ:1}
{\rm max}\{\frac{l}{2}+1, 0\} +{\rm max}\{\frac{l-2}{2}+1, 0\}={\rm max}\{l+1, 0\},\end{align}
which holds for all even numbers $l$. Here, we consult the definition of the map ``${\rm mult}$".  Similarly, if $l$ is odd, we have that $\pi^{-1}(\vec{x})=\{\frac{l-1}{2}\vec{c}+(l_1+2)\vec{z}_1+l_2\vec{z}_2+l_3\vec{z}_3, \frac{l-1}{2}\vec{c}+l_1\vec{z}_1+(l_2+2)\vec{z}_2+l_3\vec{z}_3\}$. Then the required identity in (AH2) follows by
\begin{align}\label{equ:2}
{\rm max}\{\frac{l-1}{2}+1, 0\}+{\rm max}\{\frac{l-1}{2}+1, 0\}={\rm max}\{l+1, 0\},
\end{align}
which holds for all odd numbers $l$. We are done.
\end{proof}

We compare the relations in the homogeneous coordinate algebras $S(2,2,2,2; -1)$ and $S(4,4,2)$. There is a well-defined algebra homomorphism $\phi\colon S(4,4,2)\rightarrow S(2,2,2,2;-1)$ by $\phi(z_1)=x_1$, $\phi(z_2)=x_2$ and $\phi(z_3)=x_3x_4$.

\begin{prop}\label{prop:A}
The above defined homomorphisms $\pi\colon L(4,4,2)\rightarrow L(2,2,2,2)$ and $\phi\colon S(4,4,2)\rightarrow S(2,2,2,2;-1)$ satisfy the conditions in Proposition {\rm \ref{prop:equiv}}. Consequently, we have an isomorphism of ${\rm Im}\pi$-graded algebras $$\bar{\phi}\colon \pi_*S(4,4,2)\stackrel{\sim}\longrightarrow S(2,2,2,2; -1)_{{\rm Im}\pi}$$
and an equivalences of abelian categories
$$({\rm coh}\mbox{-}\mathbb{X}(4,4,2))^{\mathbb{Z}(2\vec{\omega})}\stackrel{\sim}\longrightarrow {\rm coh}\mbox{-}\mathbb{X}(2,2,2,2;-1).$$
\end{prop}

\begin{proof}
We observe that $\phi(S(4,4,2)_{\vec{z}})\subseteq S(2,2,2,2;-1)_{\pi(\vec{z})}$ for each $\vec{z}\in L(4,4,2)$. By Lemma \ref{lem:A}, it suffices to claim that $\bar{\phi}$ is surjective, equivalently, the homogeneous component $S(2,2,2,2;-1)_{\vec{x}}$ is generated by $x_1$, $x_2$ and $x_3x_4$, whenever $\vec{x}$ lies in ${\rm Im}\pi$. We recall that such $\vec{x}$ has its normal form $l\vec{c}+l_1\vec{x}_1+l_2\vec{x}_2+l_3(\vec{x}_3+\vec{x}_4)$. Then
$\{x_1^{2a}x_2^{2b}x_1^{l_1}x_2^{l_2}(x_3x_4)^{l_3}\; |\; a+b=l, a, b\geq 0\}$ is a basis of $S(2,2,2,2;-1)_{\vec{x}}$, proving the claim. The remaining statements follow from Proposition \ref{prop:equiv}.
\end{proof}

\subsection{} In this subsection, we will relate weighted projective lines $\mathbb{X}(2,2,2,2;\epsilon)$ to $\mathbb{X}(6,3,2)$. Here, $\epsilon\in k$ satisfies $\epsilon^2-\epsilon+1=0$.  We assume further that there exists a nonzero $\Delta\in k$ satisfying $\Delta^2=6\epsilon-3$; in particular, the field $k$ is not of characteristic three. Indeed, if $k=\mathbb{C}$ is the field of complex numbers, we might take $\epsilon=\frac{1+\sqrt{-3}}{2}$ and $\Delta=\sqrt[4]{-27}$.

Recall that the string group $L(6,3,2)$ is generated by $\vec{u}_1$, $\vec{u}_2$ and $\vec{u}_3$ subject to the relations $6\vec{u}_1=3\vec{u}_2=2\vec{u}_3$. The homogeneous coordinate algebra $S(6,3,2)$ is $L(6,3,2)$-graded by means of ${\rm deg}\; u_i=\vec{u}_i$.

There is a well-defined group homomorphism $\pi\colon L(6,3,2)\rightarrow L(2,2,2,2)$ given by $\pi(\vec{u}_1)=\vec{x}_4$, $\pi(\vec{u}_2)=\vec{c}$ and $\pi(\vec{u}_3)=\vec{x}_1+\vec{x}_2+\vec{x}_3$. Here, $\vec{c}$  is the canonical element in $L(2,2,2,2)$.

\begin{lem}\label{lem:B}
The group homomorphism $\pi\colon L(6,3,2)\rightarrow L(2,2,2,2)$ is admissible with ${\rm Ker}\pi=\{0, 4\vec{u}_1+\vec{u}_2-\vec{c}, 2\vec{u}_1+2\vec{u}_2-\vec{c}\}=\mathbb{Z}(2\vec{\omega})$.
\end{lem}

Here, we observe that the dualizing element $\vec{\omega}=\vec{c}-\vec{u}_1-\vec{u}_2-\vec{u}_3$ in $L(6,3,2)$ has order six. The cyclic subgroup $\mathbb{Z}(2\vec{\omega})$ generated  by $2\vec{\omega}=4\vec{u}_1+\vec{u}_2-\vec{c}$ has order three.

\begin{proof}
The argument is similar to the proof of Lemma \ref{lem:A}. We observe that the subgroup ${\rm Im}\pi\subseteq L(2,2,2,2)$ is effective. Any element $\vec{x}$ in ${\rm Im}\pi$ has its normal form $\vec{x}=l\vec{c}+l_1(\vec{x}_1+\vec{x}_2+\vec{x}_3)+l_4\vec{x}_4$ with $l_1, l_4\in \{0, 1\}$. We compute its inverse image $\pi^{-1}(\vec{x})$ as follows.

If $3$ divides $l$, we have $\pi^{-1}(\vec{x})=\{\frac{l}{3}\vec{c}+l_4\vec{u}_1+l_1\vec{u}_3, \frac{l-3}{3}\vec{c}+(l_4+2)\vec{u}_1+2\vec{u}_2+l_1\vec{u}_3, \frac{l-3}{3}\vec{c}+(l_4+4)\vec{u}_1+\vec{u}_2+l_1\vec{u}_3\}$. This also proves the statement on the kernel of $\pi$.  If $3$ divides $l-1$, we have $\pi^{-1}(\vec{x})=\{\frac{l-1}{3}\vec{c}+l_4\vec{u}_1+\vec{u}_2+l_1\vec{u}_3, \frac{l-1}{3}\vec{c}+(l_4+2)\vec{u}_1+l_1\vec{u}_3, \frac{l-4}{3}\vec{c}+(l_4+4)\vec{u}_1+2\vec{u}_2+l_1\vec{u}_3\}$. If $3$ divides $l-2$, we have $\pi^{-1}(\vec{x})=\{\frac{l-2}{3}\vec{c}+l_4\vec{u}_1+2\vec{u}_2+l_1\vec{u}_3, \frac{l-2}{3}\vec{c}+(l_4+2)\vec{u}_1+\vec{u}_2+l_1\vec{u}_3, \frac{l-2}{3}\vec{c}+(l_4+4)\vec{u}_1+l_1\vec{u}_3\}$. In all the three cases, the identity in (AH2) follows immediately from the definition of the map ``${\rm mult}$". For example, if $3$ divides $l-1$, we use the following identity
\begin{align*}
2{\rm max}\{\frac{l-1}{3}+1, 0\}+{\rm max}\{\frac{l-4}{3}+1, 0\}={\rm max}\{l+1, 0\}.
\end{align*}
We omit the remaining details.
\end{proof}

We compute in $S(2,2,2,2;\epsilon)$ the following identity $$x_4^6=(x_2^2+(\epsilon-1)x_1^2)^3-(\Delta x_1x_2x_3)^2.$$ It follows that the algebra homomorphism $\phi\colon S(6,3,2)\rightarrow S(2,2,2,2;\epsilon)$ given by $\phi(u_1)=x_4$, $\phi(u_2)=x_2^2+(\epsilon-1)x_1^2$ and $\phi(u_3)=\Delta x_1x_2x_3$ is well defined.

\begin{prop}\label{prop:B}
The above defined homomorphisms $\pi\colon L(6,3,2)\rightarrow L(2,2,2,2)$ and $\phi\colon S(6,3,2)\rightarrow S(2,2,2,2;\epsilon)$ satisfy the conditions in Proposition {\rm \ref{prop:equiv}}.  Consequently, we have an isomorphism of ${\rm Im}\pi$-graded algebras
$$\bar{\phi}\colon \pi_*S(6,3,2)\stackrel{\sim}\longrightarrow S(2,2,2,2; \epsilon)_{{\rm Im}\pi}$$
and an equivalences of abelian categories
$$({\rm coh}\mbox{-}\mathbb{X}(6,3,2))^{\mathbb{Z}(2\vec{\omega})}\stackrel{\sim}\longrightarrow {\rm coh}\mbox{-}\mathbb{X}(2,2,2,2;\epsilon).$$
\end{prop}

\begin{proof}
This is similar to the proof of Proposition \ref{prop:A}. It suffices to claim that the homogenous component  $S(2,2,2,2; \epsilon)_{\vec{x}}$ is generated by $x_4$, $x_2^2+(\epsilon-1)x_1^2$ and $x_1x_2x_3$, whenever $\vec{x}$ lies in ${\rm Im}\pi$. Such an element $\vec{x}$ has its normal form $\vec{x}=l\vec{c}+l_1(\vec{x}_1+\vec{x}_2+\vec{x}_3)+l_4\vec{x}_4$. It follows that $\{x_1^{2a}x_2^{2b}(x_1x_2x_3)^{l_1}x_4^{l_4}\; |\; a+b=l, a, b\geq 0\}$ is a basis of $S(2,2,2,2; \epsilon)_{\vec{x}}$. Then the claim follows immediately, once we observe that $x_4^2=x_2^2-\epsilon x_1^2$ and $x_2^2+(\epsilon-1)x_1^2$ linearly span $x_1^2$ and $x_2^2$.
\end{proof}

\subsection{} In this subsection, we will relate  weighted projective lines $\mathbb{X}(3,3,3)$ to $\mathbb{X}(6,3,2)$. Here, we require that the field $k$ is not of characteristic two, and that $\sqrt{-1}$ and $\sqrt[3]{-4}$ exist in $k$.

Recall that the string group $L(3,3,3)$ is generated by $\vec{y}_1$, $\vec{y}_2$ and $\vec{y}_3$ subject to the relations $3\vec{y}_1=3\vec{y}_2=3\vec{y}_3$. The homogeneous coordinate algebra $S(3,3,3)$ is $L(3,3,3)$-graded by means of ${\rm deg}\; y_i=\vec{y}_i$.

There is a well-defined group homomorphism $\pi\colon L(6,3,2)\rightarrow L(3,3,3)$ given by $\pi(\vec{u}_1)=\vec{y}_3$, $\pi(\vec{u}_2)=\vec{y}_1+\vec{y}_2$ and $\pi(\vec{u}_3)=\vec{c}$. Here, $\vec{c}$ is the canonical element in $L(3,3,3)$.

\begin{lem}\label{lem:C}
The group homomorphism $\pi\colon L(6,3,2)\rightarrow L(3,3,3)$ is admissible with ${\rm Ker}\pi=\{0, 3\vec{u}_1+\vec{u}_3-\vec{c}\}=\mathbb{Z}(3\vec{\omega})$.
\end{lem}

Here, we observe that the dualizing element $\vec{\omega}=\vec{c}-\vec{u}_1-\vec{u}_2-\vec{u}_3$ in $L(6,3,2)$ has order six. The cyclic subgroup $\mathbb{Z}(3\vec{\omega})$ generated  by $3\vec{\omega}=3\vec{u}_1+\vec{u}_3-\vec{c}$ has order two.

\begin{proof}
The argument is similar to the proof of Lemma \ref{lem:A}.  We observe that the subgroup ${\rm Im}\pi\subseteq L(3,3,3)$ is effective. Any element $\vec{y}$ in ${\rm Im}\pi$ has its normal form $\vec{y}=l\vec{c}+l_1(\vec{y}_1+\vec{y}_2)+l_3\vec{y}_3$ with $l_1, l_3\in \{0,1,2\}$. We now describe the inverse image $\pi^{-1}(\vec{y})$.

If $l$ is even, we have $\pi^{-1}(\vec{y})=\{\frac{l}{2}\vec{c}+l_3\vec{u}_1+l_1\vec{u}_2, \frac{l-2}{2}\vec{c}+(l_3+3)\vec{u}_1+l_1\vec{u}_2+\vec{u}_3\}$. This also proves the statement for the kernel.
If $l$ is odd, we have $\pi^{-1}(\vec{y})=\{\frac{l-1}{2}\vec{c}+l_3\vec{u}_1+l_1\vec{u}_2+\vec{u}_3, \frac{l-1}{2}\vec{c}+(l_3+3)\vec{u}_1+l_1\vec{u}_2\}$. In both cases, the identity in (AH2) follows immediately from the definition of the map ``${\rm mult}$". We omit the details.
\end{proof}

We compute in $S(3,3,3)$ that $y_3^6=(y_1^3+y_2^3)^2-4(y_1y_2)^3$. Then we have a well-defined algebra homomorphism $\phi\colon S(6,3,2)\rightarrow S(3,3,3)$ given by $\phi(u_1)=y_3$, $\phi(u_2)=\sqrt[3]{-4}(y_1y_2)$ and $\phi(u_3)=\sqrt{-1}(y_1^3+y_2^3)$.

\begin{prop}\label{prop:C}
The above defined homomorphisms $\pi \colon L(6,3,2)\rightarrow L(3,3,3)$ and $\phi\colon S(6,3,2)\rightarrow S(3,3,3)$ satisfy the conditions in Proposition {\rm \ref{prop:equiv}}.  Consequently, we have an isomorphism of ${\rm Im}\pi$-graded algebras
$$\bar{\phi}\colon \pi_*S(6,3,2)\stackrel{\sim}\longrightarrow S(3,3,3)_{{\rm Im}\pi}$$
and an equivalences of abelian categories
$$({\rm coh}\mbox{-}\mathbb{X}(6,3,2))^{\mathbb{Z}(3\vec{\omega})}\stackrel{\sim}\longrightarrow {\rm coh}\mbox{-}\mathbb{X}(3,3,3).$$
\end{prop}

\begin{proof}
This is similar to the proof of Proposition \ref{prop:A}. It suffices to claim that the homogenous component  $S(3,3,3)_{\vec{y}}$ is generated by $y_3$, $y_1y_2$ and $y_1^3+y_2^3$, whenever $\vec{y}$ lies in ${\rm Im}\pi$. Such an element $\vec{y}$ has its normal form $\vec{y}=l\vec{c}+l_1(\vec{y}_1+\vec{y}_2)+l_3\vec{y}_3$. It follows that $\{y_1^{3a}y_2^{3b}(y_1y_2)^{l_1}y_3^{l_3}\; |\; a+b=l, a, b\geq 0\}$ is a basis of $S(3,3,3)_{\vec{y}}$. Then the claim follows immediately, once we observe that $y_3^3=y_2^3-y_1^3$ and $y_1^3+y_2^3$ linearly span $y_1^3$ and $y_2^3$.
\end{proof}

\subsection{} In this subsection, we give another example for Proposition \ref{prop:equiv}, which relates weighted projective lines of the same weight sequence $(2,2,2,2)$ but with different parameters. Let $\lambda\in k$ be different from $0, 1$. We consider the weighted projective line $\mathbb{X}(2,2,2,2;\lambda)$.

There is a well-defined group homomorphism $\pi\colon L(2,2,2,2)\rightarrow L(2,2,2,2)$ given by $\pi(\vec{x}_1)=\vec{x}_1+\vec{x}_3$, $\pi(\vec{x}_2)=\vec{x}_2+\vec{x}_4$ and $\pi(\vec{x}_3)=\pi(\vec{x}_4)=\vec{c}$.

\begin{lem}\label{lem:D}
The above group homomorphism $\pi\colon L(2,2,2,2)\rightarrow L(2,2,2,2)$ is admissible with ${\rm Ker}\pi=\{0, \vec{x}_3+\vec{x}_4-\vec{c}\}=\mathbb{Z}(\vec{x}_3-\vec{x}_4)$.
\end{lem}

Here, $\mathbb{Z}(\vec{x}_3-\vec{x}_4)$ denotes the cyclic subgroup generated by $\vec{x}_3-\vec{x}_4=\vec{x}_3+\vec{x}_4-\vec{c}$, which has order two.

\begin{proof}
The argument is similar to the proof of Lemma \ref{lem:A}. Any element $\vec{x}$ in ${\rm Im}\pi$ has its normal form $\vec{x}=l\vec{c}+l_1(\vec{x}_1+\vec{x}_3)+l_2(\vec{x}_2+\vec{x}_4)$ with $l_1, l_2\in \{0, 1\}$. We now describe the inverse image $\pi^{-1}(\vec{x})$.

If $l$ is even, we have $\pi^{-1}(\vec{x})=\{\frac{l}{2}\vec{c}+l_1\vec{x}_1+l_2\vec{x}_2, \frac{l-2}{2}\vec{c}+l_1\vec{x}_1+l_2\vec{x}_2+\vec{x}_3+\vec{x}_4\}$. This also proves the statement on the kernel of $\pi$. If $l$ is odd, we have $\pi^{-1}(\vec{x})=\{\frac{l-1}{2}\vec{c}+l_1\vec{x}_1+l_2\vec{x}_2+\vec{x}_3, \frac{l-1}{2}\vec{c}+l_1\vec{x}_1+l_2\vec{x}_2+\vec{x}_4\}$. In both cases, the identity in (AH2) follows immediately from the definition of the map ``${\rm mult}$". Indeed, if $l$ is even, we use (\ref{equ:1}); otherwise, we use (\ref{equ:2}).
\end{proof}

We assume that $\sqrt{1-\lambda}$ exists in $k$, and that the characteristic of $k$ is not two. We fix the choice  of $\sqrt{1-\lambda}$. Set $\xi_{\pm}=(2-\lambda)\pm 2\sqrt{1-\lambda}$. Let $\lambda'=\frac{\xi_{-}}{\xi_{+}}$, which equals $\frac{\lambda^2-8\lambda+8+4(\lambda-2)\sqrt{1-\lambda}}{\lambda^2}$. For example, if $k=\mathbb{C}$ and $\lambda=-1$, we infer that $\lambda'=17-12\sqrt{2}$.

We compute in the homogeneous coordinate algebra $S(2,2,2,2;\lambda)$ that $$(x_2x_4)^2-\xi_{\pm}(x_1x_3)^2=(x_2^2-(1\pm\sqrt{1-\lambda})x_1^2)^2.$$
We assume that $\sqrt{\xi_+}$ exists in $k$. We infer that there is a well-defined algebra homomorphism $\phi\colon S(2,2,2,2;\lambda)\rightarrow S(2,2,2,2;\lambda')$ given by $\phi(x_1)=\sqrt{\xi_+}x_1x_3$, $\phi(x_2)=x_2x_4$, $\phi(x_3)=x_2^2-(1+\sqrt{1-\lambda})x_1^2$ and $\phi(x_4)=x_2^2-(1-\sqrt{1-\lambda})x_1^2$.

\begin{prop}\label{prop:D}
Keep the notation and assumptions as above. In particular, we have $\lambda'=\frac{\xi_{-}}{\xi_{+}}$. Then the above defined homomorphisms $\pi \colon L(2,2,2,2)\rightarrow L(2,2,2,2)$ and $\phi\colon S(2,2,2,2;\lambda)\rightarrow S(2,2,2,2;\lambda')$ satisfy the conditions in Proposition {\rm \ref{prop:equiv}}.  Consequently, we have an isomorphism of ${\rm Im}\pi$-graded algebras
$$\bar{\phi}\colon \pi_*S(2,2,2,2;\lambda)\stackrel{\sim}\longrightarrow S(2,2,2,2;\lambda')_{{\rm Im}\pi}$$
and an equivalences of abelian categories
$$({\rm coh}\mbox{-}\mathbb{X}(2,2,2,2;\lambda))^{\mathbb{Z}(\vec{x}_3-\vec{x}_4)}\stackrel{\sim}\longrightarrow {\rm coh}\mbox{-}\mathbb{X}(2,2,2,2;\lambda').$$
\end{prop}

\begin{proof}
This is similar to the proof of Proposition \ref{prop:A}. It suffices to claim that the homogenous component  $S(2,2,2,2; \lambda')_{\vec{x}}$ is generated by $x_1x_3$, $x_2x_4$,  $x_2^2-(1+\sqrt{1-\lambda})x_1^2$ and $x_2^2-(1-\sqrt{1-\lambda})x_1^2$, whenever $\vec{x}$ lies in ${\rm Im}\pi$. Such an element $\vec{x}$ has its normal form $\vec{x}=l\vec{c}+l_1(\vec{x}_1+\vec{x}_3)+l_2(\vec{x}_2+\vec{x}_4)$. It follows that $\{x_1^{2a}x_2^{2b}(x_1x_3)^{l_1}(x_2x_4)^{l_2}\; |\; a+b=l, a, b\geq 0\}$ is a basis of $S(2,2,2,2;\lambda')_{\vec{x}}$. Then the claim follows immediately, once we observe that $x_2^2-(1+\sqrt{1-\lambda})x_1^2$ and $x_2^2-(1-\sqrt{1-\lambda})x_1^2$ linearly span $x_1^2$ and $x_2^2$.
\end{proof}

We conclude the paper with some remarks.

\begin{rem}
(1)  Concerning the examples for Proposition \ref{prop:equiv} we give in this paper,  all the weight sequences are of  tubular type. We do not know whether there are nontrivial examples with non-tubular weight sequences for Proposition \ref{prop:equiv}.

(2) Denote by $C_d$ the cyclic group of order $d$. We apply a general result \cite[Theorem 7.2]{El2014} about finite abelian group actions to the equivalence in  Theorem (1). It follows that there is a $C_2$-action on ${\rm coh}\mbox{-}\mathbb{X}(2,2,2,2;-1)$  such that the corresponding category of equivariant objects is equivalent to ${\rm coh}\mbox{-}\mathbb{X}(4,4,2)$. However, it seems not easy to write this $C_2$-action explicitly. Similar remarks apply to the equivalences in Theorem (2) and (3).

(3) We observe that the equivalences in Theorem (1) and (2) might be applied to the study of $\tau^2$-stable tilting complexes in \cite{Jas}. We recall that the Auslander-Reiten translation $\tau$ is induced from the degree-shift functor by the dualizing element  $\vec{\omega}$. To be more precise, $\tau^2$-stable titling complexes in ${\rm coh}\mbox{-}\mathbb{X}(4,4,2)$ (\emph{resp.} ${\rm coh}\mbox{-}\mathbb{X}(6,3,2))$) correspond to certain tilting complexes in ${\rm coh}\mbox{-}\mathbb{X}(2,2,2,2;-1)$ (\emph{resp.} ${\rm coh}\mbox{-}\mathbb{X}(2,2,2,2;\epsilon)$). Here, we use \cite[Proposition 4.5]{Chen} implicitly. Similar remarks hold for the equivalence in Theorem (3), where we study $\tau^3$-stable tilting complexes in ${\rm coh}\mbox{-}\mathbb{X}(6,3,2)$.
\end{rem}

\vskip 10pt

\noindent {\bf Acknowledgements.} J. Chen thanks Helmut Lenzing for helpful discussion on this subject during the 2014 ICRA conference in Sanya. X.W. Chen thanks Shiquan Ruan for the reference \cite{Jas}. This work is supported by National Natural Science Foundation of China (Nos. 11522113 and 11571286), and the Fundamental Research Funds for the Central Universities (No. 20720150006).

\bibliography{}

\vskip 15pt

 \noindent {\tiny

 \vskip 5pt

\noindent Jianmin Chen\\
School of Mathematical Sciences, \\
Xiamen University, Xiamen, 361005, Fujian, PR China.\\
E-mail: chenjianmin@xmu.edu.cn\\}
\vskip 3pt

{\tiny \noindent    Xiao-Wu Chen \\
 Key Laboratory of Wu Wen-Tsun Mathematics, Chinese Academy of Sciences,\\
School of Mathematical Sciences, University of Science and Technology of China,\\
No. 96 Jinzhai Road, Hefei, 230026, Anhui, P.R. China.\\
E-mail: xwchen@mail.ustc.edu.cn, URL: http://home.ustc.edu.cn/$^\sim$xwchen}

\end{document}